\documentclass{amsart}

\usepackage{mathrsfs}
\usepackage{stmaryrd}
\usepackage{amscd}
\usepackage{amssymb}
\usepackage[alphabetic]{amsrefs}
\usepackage[all]{xy}
\usepackage{mathtools}
\usepackage{bbm}

\title[Numerical Hodge standard conjecture for squares]{The numerical Hodge standard conjecture for the square of a simple abelian variety of prime dimension}
\author{Teruhisa Koshikawa}
\address{Research Institute for Mathematical Sciences, Kyoto University}
\email{teruhisa@kurims.kyoto-u.ac.jp}

\theoremstyle{plain}
\newtheorem{thm}{Theorem}[section]
\newtheorem{lem}[thm]{Lemma}
\newtheorem{prop}[thm]{Proposition}
\newtheorem{cor}[thm]{Corollary}

\theoremstyle{definition}
\newtheorem{rem}[thm]{Remark}

\newtheorem{conj}[thm]{Conjecture}

\newcommand\bC{\mathbf C}

\newcommand\bF{\mathbf F}
\newcommand\bQ{\mathbf Q}
\newcommand\bZ{\mathbf Z}

\newcommand\cA{\mathcal A}

\newcommand\cE{\mathcal E}

\newcommand\cL{\mathcal L}

\newcommand\cZ{\mathcal Z}

\DeclareMathOperator{\Aut}{Aut}

\DeclareMathOperator{\End}{End}

\DeclareMathOperator{\Sym}{Sym}

\DeclareMathOperator{\Hom}{Hom}

\DeclareMathOperator{\prim}{prim}
\DeclareMathOperator{\num}{num}

\begin{document}

\begin{abstract}
We prove the numerical Hodge standard conjecture for the square of a simple abelian variety of prime dimension, and also in some related cases.     
\end{abstract}

\maketitle

\section{Introduction}
Recently, Ancona proved the numerical Hodge standard conjecture for abelian fourfolds \cite{Ancona}. 
In fact, he proved a general theorem for certain rank 2 pure motives in mixed characteristic \cite{Ancona}*{8.1}, and showed that this general theorem is applicable to abelian fourfolds over finite fields. In this paper, we point out some other cases where the general theorem can be applied. (See \cite{Ancona}*{A.9} for another example.) 

In Ancona's work, the main cases are
\begin{enumerate}
    \item an absolutely simple abelian fourfold, and
    \item the product of a simple abelian threefold and an elliptic curve. 
\end{enumerate}
(See \cite{Ancona}*{A.8} for some discussion.)
In this paper, we generalize the second case as follows: 

\begin{thm}\label{main:example}
Let $A$ be a simple abelian variety over a field $k$. Assume either
\begin{itemize}
\item $\dim A$ is \emph{prime}, or
\item some specialization of $A$ to a finite field is absolutely simple and \emph{almost ordinary}\footnote{This means its Newton polygon is the same as the one of the product of a supersingular elliptic curve and $\dim A-1$ ordinary elliptic curves. Such a simple abelian variety exists \cite{LO}.}. 
\end{itemize}
Let $E$ be an elliptic curve. The numerical Hodge standard conjecture holds for $A\times A$ and $A\times E$. 
\end{thm}

As in \cite{Ancona}*{1.6}, this combined with \cite{Clozel} implies

\begin{cor}
The numerical equivalence on $A\times A, A\times E$ coincides with the $\ell$-adic homological equivalence on $A\times A, A\times E$ for infinitely many $\ell$. 
\end{cor}

\begin{rem}
Assume that $k$ is a finite field. 
In both cases, the Tate conjecture for $A$ is known, and, all algebraic classes come from the intersection of divisors \cites{Tate, Tankeev:prime, LZ}. Therefore, the numerical Hodge standard conjecture holds for $A$ itself and the numerical equivalence on $A$ coincides with the $\ell$-adic homological equivalence on $A$ for every $\ell$ \cite{Milne:polarization}*{3.7}, \cite{Ancona}*{Section 5}. 
However, if $\dim A\geq 3$, $A^2$ and $A \times E$ may have an exotic Tate class in the middle degree, i.e., a class that cannot be written using Tate classes of degree 2, and the Tate conjecture is not known except the case of the product of a simple threefold and an ordinary elliptic curve \cite{Milne2022}\footnote{This relies on the work of Markman \cite{Markman} on Weil classes.}. 
\end{rem}

We prove a slightly more general statement. Let $A$ be an absolutely simple abelian variety of dimension $g$ over a finite field $\bF_q$. Let $\alpha_1, \dots, \alpha_{2g}$ be the Frobenius eigenvalues of the first cohomology so that $\overline{\alpha}_i=\alpha_{i+g}$. 
Set $\beta_{i}\coloneqq q/ \alpha_i^2, 1\leq i\leq g$. Let $\Gamma'$ denote the multiplicative group generated by $\beta_i, 1\leq i \leq g$ inside $\bQ(\alpha_1, \dots, \alpha_{2g})$. The rank of $\Gamma'$ has been studied, e.g., \cites{Zarhin:nonsimple, Zarhin:eigenvalues}. Following \cite{DKZB}, we call it the \emph{angle rank} of $A$. The angle rank is always less than or equal to $g$. If the angle rank is $g$ or $A$ is a supersingular elliptic curve\footnote{The angle rank is $0=g-1$ in this case.}, all the Tate classes on $A^n$ for a positive integer $n$ can be written using Tate classes of degree 2, and the Tate conjecture holds for $A^n$.  
(This is the case for all abelian surfaces and elliptic curves.)
The converse is also true. 
Recall that such a Tate class is called \emph{Lefschetz} and a Tate class is \emph{exotic} if it is not Lefschetz. 
We are interested in the easiest case with possible exotic Tate classes:

\begin{thm}\label{main:general}
If the angle rank of $A$ is $g-1$ or $g$ and $\dim A >1$ is odd, then the numerical Hodge standard conjecture holds for $A\times A$ and $A\times E$, where $E$ is an elliptic curve. 
\end{thm}

\begin{rem}
Tankeev \cite{Tankeev:prime}*{p.332} showed that the angle rank is $g-1$ or $g$ if $g=\dim A$ is an odd prime\footnote{Tankeev excludes the case $g=3$, but the same argument actually works.}. 
Lenstra and Zarhin \cite{LZ} showed that that if $A$ is almost ordinary, the angle rank is $g-1$ when $g$ is odd and $g$ when $g$ is even; see \cite{LZ}*{6.7} (and \cite{DKZB}*{1.5}) for a slightly more general case. 
\end{rem}

\begin{rem}
If the angle rank of $A$ is $g-1$ and $E$ is ordinary, then $A\times E$ has no exotic Tate classes; see Corollary \ref{ordinary}. The same holds trivially if the angle rank of $A$ is $g$ and $E$ is supersingular. 
\end{rem}

\begin{rem}\label{even}
If we assume instead that $g$ is even and the angle rank is $g-1$ (or $g$), then we can show that the numerical Hodge standard conjecture holds for $A$ itself. This partly generalizes the case of absolutely simple abelian fourfolds in \cite{Ancona} because one can show that the angle rank is $\geq 3$ for an absolutely simple abelian fourfold if its Frobenius generates a CM field of degree $8$. 
\end{rem}

Now, Theorem \ref{main:general} clearly implies Theorem \ref{main:example}, so we will focus on Theorem \ref{main:general}.  
We shall show that $A\times A$ may have an exotic Tate class only in the middle degree, and they form a 2-dimensional space so that we can apply \cite{Ancona}*{8.1}. The case of $A\times E$ is similar. 

Finally, let us mention that we study the Tate conjecture and the Hodge standard conjecture for self-products of K3 surfaces in \cite{IIK}. 

\section{A lemma on the Hodge standard conjecture}
Let $A$ be an abelian variety of dimension $g$ over a field with a polarization $L$. 
Let $\cZ^{n}_{\num}(A)_{\bQ}$ denote the space of algebraic cycles of codimension $n$ modulo numerical equivalence.  
Recall that the Lefschetz standard conjecture holds for $A$ and $L$, and we define the primitive part $\cZ^{n, \prim}_{\num}(A)_{\bQ}$ of $\cZ^{n}_{\num}(A)_{\bQ}$. 

\begin{conj}[The numerical Hodge standard conjecture]
For a nonnegative integer $n\leq g/2$, 
The pairing
\[
\langle -, -\rangle_n \colon \cZ^{n, \prim}_{\num}(A)_{\bQ} \times \cZ^{n, \prim}_{\num}(A)_{\bQ} \to \bQ; (\alpha, \beta) \mapsto
(-1)^{n} \alpha \cdot \beta \cdot L^{g-2n}
\]
is positive definite. 
\end{conj}

Let us say that a class in $\cZ^n_{\num}(A)_{\bQ}$ is \emph{exotic} if it cannot be written as the intersection of divisors. 

\begin{lem}\label{independence}
If $A$ has exotic classes only in the middle degree, then the numerical Hodge standard conjecture is independent of $L$. 
\end{lem}

\begin{proof}
Let $\cL^{n}_{\num}(A)_{\bQ}$ denote the subspace of $\cZ^n_{\num}(A)_{\bQ}$ spanned by the intersections of divisors. 
The numerical Hodge standard conjecture is known for $\cL^{n}_{\num}(A)_{\bQ}$ by specializing to a finite field \cite{Milne:polarization}*{3.7}, \cite{Ancona}*{Section 5}. 
In particular, only the middle degree is a problem. There is an orthogonal decomposition with respect to $\langle -,- \rangle_n$
\[
\cZ^{g/2}_{\num}(A)_{\bQ}=\cL^{g/2}_{\num}(A)_{\bQ}\oplus \cE^{g/2}_{\num}(A)_{\bQ}, 
\]
where $\cE^{g/2}_{\num}(A)_{\bQ}$ is the space of exotic classes. This decomposition is independent of $L$. 
The numerical Hodge conjecture holds for $A$ and $L$ if and only if $\langle -, - \rangle_{g/2}$ is positive definite on $\cE^{g/2}_{\num}(A)_{\bQ}$, and the latter statement is independent of $L$.   
\end{proof}

\section{Exotic Tate classes}
We assume that $A$ is an absolutely simple abelian variety of dimension $g >2$ defined over a finite field $\bF_q$ of characteristic $p$. 
We use the notation $\alpha_i, \beta_i$ as in the introduction. 
Suppose first that the angle rank of $A$ is $g-1$. This implies that $\End (A)\otimes \bQ$ is a number filed of degree $2g$ generated by Frobenius. 

\begin{lem}
Assume that the angle rank is $g-1$. 
After replacing $\alpha_i$ by $\alpha_{i+g}$ if necessary, the only relation among $\beta_1, \dots \beta_g$ has the form of
\[
(\beta_1 \cdots \beta_g)^N = 1
\]
for some $N$. 
\end{lem}

\begin{proof}
Let $\beta_1^{\bZ} \cdots \beta_g^{\bZ}$ denote the free abelian group of rank $g$ with the basis $\beta_1, \dots, \beta_g$, and let $\Gamma_1$ be the kernel of the natural map
\[
\beta_1^{\bZ} \cdots \beta_g^{\bZ} \to \bQ(\alpha_1, \dots, \alpha_{2g})\setminus\{0\}. 
\]
By assumption, $\Gamma_1$ is a free abelian group of rank 1. 
So, the Galois group of $\bQ(\alpha_1, \dots, \alpha_{2g})$ acts naturally on $\Gamma_1$ via $\{\pm 1\}\subset \Aut (\Gamma_1)$. Note that the Galois group acts on $\{\{\beta_1^{\pm 1}\}, \dots, \{\beta_g^{\pm 1}\}\}$ by permutation and the action is transitive, and the Galois group contains the complex conjugation so that $\overline{\beta}_i=\beta_i^{-1}$. 
This implies that a generator of $\Gamma_1$ has the form of
\[
\beta_1^{\pm N} \beta_2^{\pm N} \cdots \beta_g^{\pm N}
\]
for some $N$. 
\end{proof}

\begin{cor}
Let $\ell$ be a prime different from $p$. 
If $g$ is odd (resp. even), any exotic $\ell$-adic Tate class of $A\times A$, $A\times E$ (resp. $A$) is in the middle degree. If an exotic Tate class exists, then the space of exotic Tate classes is two-dimensional for $A\times A$ (resp. $A$) and four-dimensional for $A\times E$. 
\end{cor}

\begin{cor}\label{ordinary}
If $g$ is odd and $E$ is ordinary, then $A\times E$ has no exotic Tate classes. 
\end{cor}

A similar argument shows the following:

\begin{lem}
Suppose the angle rank of $A$ equals $g$ and $g$ is odd, then any exotic $\ell$-adic Tate class of $A\times E$ is in the middle degree. If an exotic Tate class exists, then $E$ is ordinary and the space of exotic Tate classes is two-dimensional. 
\end{lem}

Next, we construct a motivic counterpart of possible exotic Tate classes using complex multiplication. 
Let us first recall some facts about the motive of $A$ \cite{Ancona}*{Section 4, Section 6}. 
Set $B\coloneqq \End (A)\otimes\bQ$ and write $L\subset \overline{\bQ}$ for the Galois closure of $B$ with $\Sigma\coloneqq \Hom (B, L)$. As in \cite{Ancona}*{6.6}, there is the following decomposition in the category of Chow motives with coefficients in $L$:
\[
[H^1 (A)]=\bigoplus_{\sigma \in \Sigma} M_{\sigma}
\]
that induces \cite{Ancona}*{6.7 (1)}
\[
[H^g (A)]=\bigoplus_{I\subset \Sigma, \# I =g} M_I,   
\]
where $M_I=\otimes_{i \in I} M_i$.  
This further induces the following decomposition \cite{Ancona}*{6.7 (2)}, in the category of Chow motives with coefficients in $\bQ$, 
\[
[H^g (A)]=\bigoplus M_{[I]}, 
\]
where $[I]$ denotes the Galois orbit of $I$ and $M_{[I]}$ is the direct sum of $M_?$ over the Galois orbit.  
This decomposition is orthogonal as numerical motives with respect to $\langle -, -\rangle^{\otimes g}_{1, \textnormal{mot}} $ defined in \cite{Ancona}*{3.6}. 
Similarly, $[H^{2g}(A\times A)]$ has such a decomposition and we have summands like
\[
M_{I^2}\coloneqq M_I\otimes M_{I}, \quad M_{[I^2]}. 
\]

\begin{prop}
Assume that the angle rank is $g-1$. 
\begin{enumerate}
    \item If $g$ is even, there exists at most one $[I]$ such that the $\ell$-adic realization of $M_{[I]}$ is exotic. The numerical algebraic classes in $M_{[I]}$ is zero or two-dimensional. 
    \item If $g$ is odd and $E$ is supersingular, there exists at most one $[I]$ such that the $\ell$-adic realization of $M_{[I]}\otimes H^1 (E)$ is exotic. The numerical algebraic classes in $M_{[I]}$ is zero or four-dimensional. 
    \item If $g$ is odd, there exists at most one $[I^2]$ such that the $\ell$-adic realization of $M_{[I^2]}$ is exotic. The numerical algebraic classes in $M_{[I^2]}$ is zero or two-dimensional.  
\end{enumerate}
\end{prop}

\begin{proof}
This follows from the description of exotic Tate classes and \cite{Ancona}*{6.8}. The key claim here is that the relevant Galois orbit only has two elements, and it controls the dimension of numerical algebraic classes.  
\end{proof}

We call $M_{[I]}, M_{[I]}\otimes H^1(E), M_{[I^2]}$ \emph{exotic} if it has a nonzero numerical algebraic class. 
If it is the case, their $\ell$-adic realizations are the only exotic Tate classes. 
By \cite{Ancona}*{5.3} and Lemma \ref{independence}, the numerical Hodge standard conjecture for $A, A\times A, A\times E$ reduces to the corresponding problem on $M_{I}, M_{[I]}\otimes H^1(E), M_{[I^2]}$ respectively, with respect to $\langle -, -\rangle^{\otimes g}_{1, \textnormal{mot}}, \langle -, -\rangle^{\otimes g+1}_{1, \textnormal{mot}}. \langle -, -\rangle^{\otimes 2g}_{1, \textnormal{mot}}$ for some polarization. 

A similar construction makes sense for $A \times E$ if $g$ is odd, the angle rank is $g$, and $E$ is ordinary.  

Finally, when $g$ is odd and $E$ is supersingular, an exotic $M_{[I]}\otimes H^1(E)$ has a decomposition into rank 2 motives
\[
M_{[I]}\otimes H^1(E) =M_1 \oplus M_2
\]
orthogonal with respect to $\langle -, -\rangle^{\otimes g+1}_{1, \textnormal{mot}}$. 
More precisely, the Galois action on $[I]$ gives rise to an imaginary quadratic field $F$ inside $B$ and there is an embedding $F\hookrightarrow \End (E_{\overline{\bF}_q})\otimes \bQ$ by exactly the same argument as in the proof of \cite{Ancona}*{7.16}. The actions of $F$ on $M_{[I]}$ and $H^1(E)$ induce the above decomposition. 

\section{Ancona's theorem for rank 2 motives}
To conclude the proof of Theorem \ref{main:general}, we recall Ancona's theorem and then use CM liftings to apply it. 

Let $K$ be a $p$-adic field with the ring of integers $O_K$ with residue field $k$. 
Fix an embedding $\sigma \colon K\hookrightarrow \bC$. 
We shall use the language of relative Chow motives over $O_K$, equipped with base changes to $\bC$ via $\sigma$ and to $k$ via the specialization.   
For a relative Chow motive $M$ over $O_K$, we write $V_B$ for the Betti realization of $M_{\bC}$. 
Let $V_Z$ denote the space of numerical algebraic cycles in $M_k$, i.e., homomorphisms from $\mathbbm{1}$ modulo numerical equivalences. Both $V_B$ and $V_Z$ are $\bQ$-vector spaces. 
If $M$ has a quadratic form
\[
q\colon \Sym^2 (M) \to \mathbbm{1}, 
\]
then it induces ($\bQ$-valued) quadratic forms $q_B, q_Z$ on $V_B, V_Z$ respectively. 

\begin{thm}[Ancona \cite{Ancona}*{8.1}]\label{rank 2}
Let $M$ be a relative Chow motive over $O_K$ with a quadratic form $q$. 
Assume that
\begin{itemize}
    \item $\dim_{\bQ} V_B=\dim_{\bQ} V_Z=2$, and 
    \item $q_B\colon V_B\times V_B \to \bQ$ is a polarization of Hodge structures. 
\end{itemize}
Then, $q_Z$ is positive definite. 
\end{thm}

\begin{proof}[Proof of Theorem \ref{main:general}]
Following the proof of \cite{Ancona}*{3.18}, we use CM liftings to prove Theorem \ref{main:general}. 
Let $A$ be as in Theorem \ref{main:general}. 
Set $B\coloneqq \End (A_{\bF_q})\otimes \bQ$.  
After enlarging $\bF_q$, we can find a finite extension $O_K$ of $W(\bF_q)$ and an abelian scheme $\cA$ over $O_K$ with $B\to \End (\cA)$ such that the reduction $\cA_{\bF_q}$ is $B$-isogenous to $A$. 
We may replace $A$ by $\cA_{\bF_q}$ and assume that a polarization on $\cA_{\bF_q}$ lifts to a polarization on $\cA$. 

If $A^2$ has no exotic classes, there is nothing to prove. So, assume some $M_{[I^2]}$ is exotic. 
By \cite{Ancona}*{5.3}, it suffices to show the the paring $\langle -, - \rangle^{\otimes 2g}_{1, \textnormal{mot}}$ is positive definite on the exotic $M_{[I^2]}$. By the construction of $M_{[I^2]}$ and the paring, it lifts to a relative Chow motive with a quadratic form over $O_K$ (cf. \cite{Ancona}*{4.1, 4.2} and references therein, and the proof of \cite{Ancona}*{3.18}). 
By the definition of the exotic $M_{[I^2]}$, this lift satisfies the assumption of Theorem \ref{rank 2}. 
So, $\langle -, - \rangle^{\otimes 2g}_{1, \textnormal{mot}}$ is positive definite on $M_{[I^2]}$. 

The case of $A\times E$ is similar as in \cite{Ancona}. Let us consider the case $E$ is supersingular. The decomposition
\[
M_{[I]}\otimes H^1(E) =M_1 \oplus M_2
\]
is constructed using the action of the imaginary quadratic field $F\subset B$ on $E$, and it may also lifts by taking a lift of $E$ with the action of $F$. 
\end{proof}

\begin{rem}
Consider the case $A\times A$. 
The Hodge type of the Betti realization of the lifts of the exotic classes have the form of $(2a, 2b), (2b. 2a)$ with $a+b=g$. In particular, it is never $(g,g)$ and any exotic class cannot be lifted to an algebraic class of $\cA_{\bC} \times \cA_{\bC}$. Therefore, Ancona's theorem is essential. 
\end{rem}

Remark \ref{even} can be proved in the same way.

\end{document}